\documentclass[12pt]{amsart}

\usepackage{enumitem}
\usepackage[utf8]{inputenc}
\usepackage[T1]{fontenc}
\usepackage[english]{babel}
\usepackage{amssymb}
\usepackage{array}
\usepackage{mathtools}
\usepackage[svgnames]{xcolor}
\usepackage{hyperref}
\usepackage{mathrsfs}
\usepackage{latexsym}
\usepackage{amsfonts}
\usepackage{amsmath}
\usepackage{amsthm}
\numberwithin{equation}{section}

\def\CC{\mathbb{C}}
\def\NN{\mathbb{N}}

\def\ZZ{\mathbb{Z}}
\def\Re{\mathfrak{R}}
\def\Im{\mathfrak{I}}

\newtheorem{Proposition}{Proposition}[section]

\newtheorem{Lemma}[Proposition]{Lemma}

\newtheorem{Corollary}[Proposition]{Corollary}
\newtheorem{Remark}[Proposition]{Remark}

\def\Li{\operatorname{Li}}

\title{The dyadic Cauchy-kernel identity:\\ several roads back to classical objects}
\author{Nick Castillo}
\date{\today}

\begin{document}

\maketitle

\begin{abstract}
This is an expository note. We take the dyadic Cauchy-kernel identity of Castillo--Costin--Costin \cite{GRA} --- a global rational/factorial decomposition built on the polylogarithm --- and follow it down several specializations. In each direction it returns to a classical landmark: the polylogarithm duplication formula and Hurwitz's Fourier-series formula (\S\ref{sec:sine}); representations of $\zeta$ at the special argument $\pi$ and at rational arguments, in the neighborhood of Hurwitz's multiplication theorem (\S\ref{sec:zeta}); the Hasse--Sondow globally convergent series (\S\ref{sec:mellin}); and, through its discrete scale invariance, the extra zeros of the Dirichlet eta function together with the harmonic-sum asymptotics of Flajolet--Gourdon--Dumas, with Dirichlet $L$-values emerging as the amplitudes of a log-periodic oscillation (\S\ref{sec:Lvalues}). The aim is unification: to exhibit one compact identity as an organizing center from which these classical results may be read off. We claim no new theorems; where an identity may not previously have been displayed in exactly this form, we say so and explain why it is nonetheless a recombination of known ingredients.
\end{abstract}

\section{Introduction}\label{sec:intro}

Our starting point is the following identity, established in \cite{GRA}: for $\Re p>0$ and $0<\Re s<1$,
\begin{equation}\label{GRAid}
    \pi p^{s-1}=\Gamma(s)\sin(\pi s)\left[\Li_s\!\left(e^{-p}\right)-\sum_{k=1}^\infty 2^{-k(1-s)}\Li_s\!\left(-e^{-2^{-k}p}\right)\right],
\end{equation}
where $\Li_s(z)=\sum_{n\ge1}z^n/n^s$ is the polylogarithm. We refer to \eqref{GRAid} as the \emph{dyadic Cauchy-kernel identity}: it expresses the homogeneous function $p^{s-1}$ through a polylogarithm at $e^{-p}$ corrected by a geometrically weighted sum of polylogarithms at the dyadically contracted, sign-reversed points $-e^{-2^{-k}p}$.

It is worth recording at the outset what \eqref{GRAid} \emph{is}, because it organizes everything that follows.

\begin{Lemma}\label{lem:telescope}
For $\Re p>0$ and $0<\Re s<1$, the bracket in \eqref{GRAid} telescopes:
\begin{equation}\label{telescope}
    \Li_s\!\left(e^{-p}\right)-\sum_{k=1}^\infty 2^{-k(1-s)}\Li_s\!\left(-e^{-2^{-k}p}\right)=\lim_{k\to\infty}2^{-k(1-s)}\Li_s\!\left(e^{-2^{-k}p}\right)=\Gamma(1-s)\,p^{s-1}.
\end{equation}
In particular, modulo the reflection formula $\Gamma(s)\Gamma(1-s)\sin(\pi s)=\pi$, the dyadic Cauchy-kernel identity is the iterated polylogarithm duplication formula, telescoped.
\end{Lemma}

\begin{proof}
The duplication (``square'') formula for the polylogarithm \cite{Lewin} reads $\Li_s(z)+\Li_s(-z)=2^{1-s}\Li_s(z^2)$. With $z=e^{-2^{-k}p}$, so that $z^2=e^{-2^{-(k-1)}p}$,
\begin{equation}
    \Li_s\!\left(-e^{-2^{-k}p}\right)=2^{1-s}\Li_s\!\left(e^{-2^{-(k-1)}p}\right)-\Li_s\!\left(e^{-2^{-k}p}\right).
\end{equation}
Writing $b_k:=2^{-k(1-s)}\Li_s(e^{-2^{-k}p})$ and multiplying by $2^{-k(1-s)}$ gives $2^{-k(1-s)}\Li_s(-e^{-2^{-k}p})=b_{k-1}-b_k$, so the dyadic sum telescopes:
\begin{equation}
    \sum_{k=1}^\infty 2^{-k(1-s)}\Li_s\!\left(-e^{-2^{-k}p}\right)=\sum_{k=1}^\infty(b_{k-1}-b_k)=b_0-\lim_{k\to\infty}b_k=\Li_s\!\left(e^{-p}\right)-\lim_{k\to\infty}b_k.
\end{equation}
The bracket therefore equals $\lim_k b_k$. Finally, the standard boundary expansion $\Li_s(e^{-\varepsilon})=\Gamma(1-s)\varepsilon^{s-1}+\sum_{m\ge0}\zeta(s-m)(-\varepsilon)^m/m!$ as $\varepsilon\to0^+$ gives, with $\varepsilon=2^{-k}p$,
\begin{equation}
    \lim_{k\to\infty}b_k=\lim_{k\to\infty}2^{-k(1-s)}\Gamma(1-s)(2^{-k}p)^{s-1}=\Gamma(1-s)\,p^{s-1},
\end{equation}
since $2^{-k(1-s)}2^{-k(s-1)}=1$. Combining with \eqref{GRAid} and the reflection formula recovers $\pi p^{s-1}=\pi p^{s-1}$.
\end{proof}

Lemma~\ref{lem:telescope} is the expository thesis in miniature: the dyadic identity is built from a single classical ingredient (duplication) plus a boundary asymptotic. The sections that follow specialize \eqref{GRAid} in various ways, and each time the same phenomenon recurs --- a classical object emerges. We have verified the numerical identities reported below to high precision; we indicate this where relevant.

\section{The dyadic Clausen sine series and Hurwitz's formula}\label{sec:sine}

Setting $p=\pm i\omega$ in \eqref{GRAid} and symmetrizing yields a real Fourier-type series.

\begin{Proposition}\label{Prop1}
    Given $\omega>0$ with $\omega\notin 2\pi\ZZ$ and $s\in\CC$ with $0<\Re s<1$,
    \begin{equation}\label{Clausian}
    \omega^{s-1}=\frac{1}{\Gamma(1-s)}\sum_{n=1}^{\infty}\frac{1}{n^s}\left\{\sin\!\left(\frac{\pi s}{2}+n\omega \right) -(-1)^n\sum_{k=1}^{\infty} \frac{\sin\!\left(\frac{\pi s}{2}+\frac{n\omega}{2^k}\right)}{2^{k(1-s)}}\right\}.
    \end{equation}
\end{Proposition}
\begin{proof}
    Recall the Clausen functions $C(\omega,s)=\sum_n\cos(n\omega)/n^s$ and $S(\omega,s)=\sum_n\sin(n\omega)/n^s$, so that $\Li_s(e^{\mp i\omega})=C(\omega,s)\mp iS(\omega,s)$ and $\Li_s(-e^{\mp i\omega2^{-k}})=\sum_n(-1)^n n^{-s}e^{\mp in\omega2^{-k}}$. Evaluating \eqref{GRAid} at $p=i\omega$ and at $p=-i\omega$ and expanding the polylogarithms as power series gives the conjugate pair
\begin{align}
    \pi (i\omega)^{s-1}&=\Gamma(s)\sin(\pi s)\sum_{n=1}^\infty \frac{1}{n^s}\Big\{ e^{-in\omega}-\sum_{k=1}^\infty 2^{-k(1-s)}(-1)^n e^{-in\omega 2^{-k}}\Big\},\label{piw}\\
    \pi (-i\omega)^{s-1}&=\Gamma(s)\sin(\pi s)\sum_{n=1}^\infty \frac{1}{n^s}\Big\{ e^{in\omega}-\sum_{k=1}^\infty 2^{-k(1-s)}(-1)^n e^{in\omega 2^{-k}}\Big\}.\label{miw}
\end{align}
Since $(\pm i\omega)^{s-1}=e^{\pm i\pi(s-1)/2}\omega^{s-1}$, multiply \eqref{piw} by $e^{-i\pi(s-1)/2}$ and \eqref{miw} by $e^{i\pi(s-1)/2}$, so that each left-hand side becomes $\pi\omega^{s-1}$. Using $\Gamma(s)\sin(\pi s)=\pi/\Gamma(1-s)$ and the elementary identity
\begin{equation}\label{expansion}
    e^{\mp i\pi(s-1)/2}e^{\mp in\varphi}=\sin\!\left(\tfrac{\pi s}{2}+n\varphi\right)\pm i\cos\!\left(\tfrac{\pi s}{2}+n\varphi\right),
\end{equation}
and averaging the two resulting identities, the cosine series --- single and dyadic alike --- cancel, leaving
\begin{equation}\label{AfRefProp1Id}
    \omega^{s-1}=\frac{1}{\Gamma(1-s)}\Big\{\sum_{n=1}^\infty \frac{1}{n^s}\sin\!\left(\tfrac{\pi s}{2}+n\omega\right)-\sum_{k=1}^{\infty}2^{-k(1-s)}\sum_{n=1}^\infty \frac{(-1)^n}{n^s}\sin\!\left(\tfrac{\pi s}{2}+\tfrac{n\omega}{2^k}\right)\Big\}.
\end{equation}
The hypothesis $\omega\notin2\pi\ZZ$ guarantees convergence of the Clausen series and that $\omega2^{-k}\not\equiv\pi\pmod{2\pi}$ for every $k$. It remains to justify the interchange of summations, for which we invoke the Moore--Osgood theorem \cite{Graves}. By Abel summation, for $z\in\CC\setminus\{1\}$ with $|z|=1$,
\begin{equation}\label{AbelEst}
    \left|\sum_{n=1}^N\frac{z^n}{n^s}\right|\leq \frac{C_s}{|1-z|},\qquad C_s=2\left(1+\frac{|s|}{\Re s}\right),
\end{equation}
uniformly in $N$. Writing the inner sine as exponentials with $\theta=\omega2^{-k}$ and applying \eqref{AbelEst} at $z=-e^{\pm i\theta}$,
\begin{equation}\label{InnernSumEst}
    \left|\sum_{n=1}^N \frac{(-1)^n}{n^s}\sin\!\left(\tfrac{\pi s}{2}+n\theta\right)\right|\leq \frac{C_s\cosh\!\left(\tfrac{\pi \Im s}{2}\right)}{|1+e^{i\theta}|}.
\end{equation}
With $\theta_k=\omega2^{-k}\to0$ we have $|1+e^{i\theta_k}|\to2$ and each is positive, so $\inf_k|1+e^{i\theta_k}|>0$; hence
\begin{equation}
    M_s:=\sup_{(k,N)\in\NN^2}\left|\sum_{n=1}^N \frac{(-1)^n}{n^s}\sin\!\left(\tfrac{\pi s}{2}+n\theta_k\right)\right|\leq \frac{C_s\cosh\!\left(\tfrac{\pi \Im s}{2}\right)}{\inf_{k}|1+e^{i\theta_k}|}<\infty.
\end{equation}
For each fixed $K$ the inner limit in $N$ exists by the Dirichlet test \cite{RudinPrinc}, and writing $a_{n,k}=(-1)^n n^{-s}\sin(\tfrac{\pi s}{2}+n\theta_k)$, $S_{N,K}=\sum_{k\le K}\sum_{n\le N}2^{-k(1-s)}a_{n,k}$,
\begin{equation}
        |S_{N,\infty}-S_{N,K}|\leq M_s \sum_{k=K+1}^\infty 2^{-k(1-\Re s)}=\frac{M_s\, 2^{-K(1-\Re s)}}{2^{\,1-\Re s}-1},
\end{equation}
uniformly in $N$. The two iterated limits therefore agree, which is the required interchange.
\end{proof}

\begin{Remark}[what is classical here]\label{rem:hurwitz}
The single sum in \eqref{Clausian} is, term for term, Hurwitz's Fourier-series formula for the Hurwitz zeta function. Indeed, with $a=\omega/2\pi$,
\begin{equation}\label{singleHurwitz}
    \sum_{n=1}^\infty\frac{1}{n^s}\sin\!\left(\tfrac{\pi s}{2}+n\omega\right)=(2\pi)^{s-1}\Gamma(1-s)\big[\zeta(1-s,a)-\cos(\pi s)\,\zeta(1-s,1-a)\big],
\end{equation}
which is Hurwitz's formula \cite{Apostol,Titchmarsh}, valid throughout $0<\Re s<1$ by the extension of Boudjelkha \cite{Boudjelkha}. By Lemma~\ref{lem:telescope} the dyadic correction is the telescoped duplication formula. Thus Proposition~\ref{Prop1} combines two classical ingredients --- Hurwitz's formula and the duplication formula --- and produces no transcendental content beyond them.
\end{Remark}

\section{Specializations to the zeta function}\label{sec:zeta}

\subsection{The argument $\omega=\pi$.} Specializing \eqref{Clausian} to $\omega=\pi$ yields a representation of $\pi^{s-1}$.
\begin{Corollary}\label{piRep}
    For $0<\Re s<1$,
    \begin{equation}\label{piRepEq}
        \pi^{s-1} = \frac{1}{\Gamma(1-s)}\left\{- \sin\!\left(\tfrac{\pi s}{2}\right)(1-2^{1-s})\zeta(s) -\sum_{n=1}^{\infty}\sum_{k=1}^{\infty} \frac{(-1)^n}{n^s} \frac{\sin\!\left(\tfrac{\pi s}{2}+\tfrac{n\pi}{2^k}\right)}{2^{k(1-s)}}\right\}.
    \end{equation}
\end{Corollary}
\begin{proof}
    At $\omega=\pi$, the angle-addition formula gives $\sin(\tfrac{\pi s}{2}+n\pi)=(-1)^n\sin(\tfrac{\pi s}{2})$, so the single sum in \eqref{Clausian} factors as $\sin(\tfrac{\pi s}{2})\sum_n(-1)^n n^{-s}$. The alternating Dirichlet series satisfies $\sum_n(-1)^{n-1}n^{-s}=(1-2^{1-s})\zeta(s)$ (immediate for $\Re s>1$ from $\zeta(s)-2\cdot2^{-s}\zeta(s)$, and valid in the strip by continuation), whence $\sum_n(-1)^n n^{-s}=-(1-2^{1-s})\zeta(s)$. The double sum is carried over unchanged. No appeal to the functional equation is required.
\end{proof}

Write $\eta(s)=(1-2^{1-s})\zeta(s)$ for the Dirichlet eta function and
\begin{equation}
    \mathcal{D}(s):=\sum_{n=1}^{\infty}\sum_{k=1}^{\infty} \frac{(-1)^n}{n^s} \frac{\sin\!\left(\tfrac{\pi s}{2}+\tfrac{n\pi}{2^k}\right)}{2^{k(1-s)}}
\end{equation}
for the dyadic double sum. Combining Corollary~\ref{piRep} with Riemann's functional equation gives a representation of the ratio of $\zeta$ to its reflection.

\begin{Proposition}\label{ratioRep}
    For $0<\Re s<1$, away from zeros of $\zeta(1-s)$,
    \begin{equation}
        \frac{\zeta(s)}{\zeta(1-s)}=-2^s \sin\!\left(\tfrac{\pi s}{2}\right)\left[\sin\!\left(\tfrac{\pi s}{2}\right) \eta(s) +\mathcal{D}(s)\right].
    \end{equation}
\end{Proposition}
\begin{proof}
    By Corollary~\ref{piRep}, $\Gamma(1-s)\pi^{s-1}=-[\sin(\tfrac{\pi s}{2})\eta(s)+\mathcal{D}(s)]$. Riemann's functional equation in the form $\zeta(s)=2^s\pi^{s-1}\sin(\tfrac{\pi s}{2})\Gamma(1-s)\zeta(1-s)$ gives $\zeta(s)/\zeta(1-s)=2^s\sin(\tfrac{\pi s}{2})(\pi^{s-1}\Gamma(1-s))$; substituting the previous display yields the assertion as an identity of meromorphic functions.
\end{proof}

\begin{Remark}
Proposition~\ref{ratioRep} is equivalent to Corollary~\ref{piRep}: substituting $\mathcal{D}(s)=-\pi^{s-1}\Gamma(1-s)-\sin(\tfrac{\pi s}{2})\eta(s)$ collapses it identically to Riemann's functional equation. It is recorded only because the ratio form is suggestive; it carries no information beyond Corollary~\ref{piRep}.
\end{Remark}

\subsection{Rational arguments.} Write $\mathcal{D}(s;p/q)$ for the dyadic double sum at $\omega=2\pi p/q$,
\begin{equation}
    \mathcal{D}(s;p/q):=\sum_{n=1}^{\infty}\sum_{k=1}^{\infty}  \frac{(-1)^n\sin\!\left(\tfrac{\pi s}{2}+\tfrac{2 n\pi p}{q 2^k }\right)}{n^s 2^{k(1-s)}}.
\end{equation}
Specializing \eqref{Clausian} to $\omega=2\pi p/q$ and invoking Hurwitz's formula gives a representation at rational arguments.
\begin{Proposition}\label{ratRep}
    Let $p,q\in\NN$ with $0<p<q$ and $0<\Re s<1$. Then
    \begin{equation}
        \left(\frac{p}{q}\right)^{s-1}=\zeta\!\left(1-s,\tfrac{p}{q}\right)-\cos(\pi s)\,\zeta\!\left(1-s,1-\tfrac{p}{q}\right)-\frac{\mathcal{D}(s;p/q)}{(2\pi)^{s-1}\Gamma(1-s)}.
    \end{equation}
\end{Proposition}
\begin{proof}
    Set $a:=p/q\in(0,1)$; then $\omega=2\pi a\notin2\pi\ZZ$, so Proposition~\ref{Prop1} applies. Its double sum is $\mathcal{D}(s;a)$, and writing $\omega^{s-1}=(2\pi)^{s-1}(p/q)^{s-1}$,
    \begin{equation}\label{ratStart}
        \left(\tfrac{p}{q}\right)^{s-1}=\frac{\mathcal{T}(s;a)-\mathcal{D}(s;a)}{(2\pi)^{s-1}\Gamma(1-s)},\qquad \mathcal{T}(s;a):=\sum_{n=1}^\infty\frac{1}{n^s}\sin\!\left(\tfrac{\pi s}{2}+2\pi n a\right).
    \end{equation}
    With the periodic zeta function $F(x,s):=\sum_n e^{2\pi inx}/n^s$ and $c:=e^{i\pi s/2}$, expanding the sine gives $\mathcal{T}(s;a)=\tfrac{1}{2i}(c\,F(a,s)-c^{-1}F(-a,s))$. Hurwitz's formula \cite{Apostol}, established for $\Re s>1$, $0<a\le1$ and continued to $0<\Re s<1$, $0<a<1$ via the conditional convergence of $F(\pm a,s)$ (Dirichlet's test, $a\notin\ZZ$) and the continuation of the Hurwitz zeta, reads
    \begin{equation}\label{Hurwitz}
        \zeta(1-s,a)=\frac{\Gamma(s)}{(2\pi)^s}\big(c^{-1}F(a,s)+c\,F(-a,s)\big).
    \end{equation}
    Applying \eqref{Hurwitz} at $a$ and $1-a$ (using $F(-a,s)=F(1-a,s)$ and $F(a-1,s)=F(a,s)$) yields the system $c^{-1}A+cB=Z_1$, $cA+c^{-1}B=Z_2$, with $A=F(a,s)$, $B=F(-a,s)$, $Z_1=\tfrac{(2\pi)^s}{\Gamma(s)}\zeta(1-s,a)$, $Z_2=\tfrac{(2\pi)^s}{\Gamma(s)}\zeta(1-s,1-a)$. Solving and using $c^2\pm c^{-2}=2\cos\pi s,\,2i\sin\pi s$,
    \begin{equation}
        \mathcal{T}(s;a)=\tfrac{1}{2i}(cA-c^{-1}B)=\frac{Z_1-\cos(\pi s)Z_2}{2\sin\pi s}=\frac{(2\pi)^s}{2\sin(\pi s)\Gamma(s)}\big[\zeta(1-s,a)-\cos(\pi s)\zeta(1-s,1-a)\big].
    \end{equation}
    Dividing by $(2\pi)^{s-1}\Gamma(1-s)$ and using $\Gamma(s)\Gamma(1-s)=\pi/\sin(\pi s)$ collapses the prefactor to $1$, so $\mathcal{T}(s;a)/[(2\pi)^{s-1}\Gamma(1-s)]=\zeta(1-s,a)-\cos(\pi s)\zeta(1-s,1-a)$. Substituting into \eqref{ratStart} gives the claim.
\end{proof}

\begin{Remark}
Proposition~\ref{ratRep} is Hurwitz's formula at rational arguments, with the dyadic term as a determined remainder. It lies adjacent to the Hurwitz multiplication theorem $k^s\zeta(s)=\sum_{n=1}^k\zeta(s,n/k)$, which already supplies rational-argument relations such as $\zeta(s,\tfrac12)=(2^s-1)\zeta(s)$.
\end{Remark}

\section{A companion observation: the Mellin route and Hasse--Sondow}\label{sec:mellin}

There is a second, independent way the dyadic machinery meets the zeta function. Applying the global factorial expansion of \cite{GRA} to the Mellin representation
\begin{equation}
    \zeta(s)=\frac{1}{s-1}+\frac{\sin(\pi s)}{\pi}\int_0^\infty\big(\ln(1+x)-\Psi(1+x)\big)x^{-s}\,dx \qquad(0<\Re s<1)
\end{equation}
produces a convergent double series of Mellin transforms of rational functions. Evaluating those integrals in closed form and summing the resulting geometric series in the dyadic index collapses the expansion to
\begin{equation}\label{HS}
    \zeta(s)=-\frac{2^{s-1}}{1-2^{s-1}}\,\eta_{\mathrm{HS}}(s),\qquad \eta_{\mathrm{HS}}(s)=\sum_{n=0}^\infty\frac{1}{2^{n+1}}\sum_{m=0}^n(-1)^m\binom{n}{m}(m+1)^{-s},
\end{equation}
in which $\eta_{\mathrm{HS}}$ is precisely the Hasse--Sondow globally convergent series for $\eta$ \cite{Hasse,Sondow}; one checks directly that $-2^{s-1}/(1-2^{s-1})$ times $(1-2^{1-s})\zeta(s)=\eta(s)$ equals $\zeta(s)$. We have verified \eqref{HS} numerically, including on the critical line and at the first nontrivial zero. Thus the Mellin route, like the Fourier route of \S\ref{sec:sine}, recovers a classical convergent series; the finite-difference inner sum is the Euler-transformation (N\"orlund--Rice) structure underlying Hasse--Sondow.

\section{Discrete scale invariance and Dirichlet $L$-values}\label{sec:Lvalues}

The dyadic weight $2^{-k(1-s)}$ carries a scaling structure that we now isolate, following the analytic theory of harmonic sums \cite{FGD}. For a profile $\phi:(0,\infty)\to\CC$ of suitable decay, set
\begin{equation}\label{harmonic}
    F_\phi(\omega):=\sum_{k=1}^\infty 2^{-k(1-s)}\phi\!\left(\omega 2^{-k}\right);
\end{equation}
the dyadic correction of Proposition~\ref{Prop1} is $F_\phi$ with profile $\phi(\theta)=\sum_n(-1)^n n^{-s}\sin(\tfrac{\pi s}{2}+n\theta)$. Two structural facts govern $F_\phi$.

\emph{Renormalization.} A reindexing gives the exact functional equation
\begin{equation}\label{RG}
    F_\phi(2\omega)=2^{s-1}\big[\phi(\omega)+F_\phi(\omega)\big],
\end{equation}
the statement of dyadic discrete scale invariance: up to the boundary term, doubling $\omega$ multiplies $F_\phi$ by $2^{s-1}$.

\emph{Complex dimensions.} Taking the Mellin transform $\widetilde F_\phi(\sigma)=\int_0^\infty F_\phi(\omega)\omega^{\sigma-1}d\omega$ and using $\int_0^\infty\phi(\omega2^{-k})\omega^{\sigma-1}d\omega=2^{k\sigma}\widetilde\phi(\sigma)$,
\begin{equation}\label{MellinF}
    \widetilde F_\phi(\sigma)=\frac{\widetilde\phi(\sigma)}{2^{(1-s)-\sigma}-1},
\end{equation}
whose denominator vanishes exactly at the tower
\begin{equation}\label{tower}
    \sigma_n=(1-s)-\frac{2\pi i n}{\ln 2},\qquad n\in\ZZ .
\end{equation}
These are the complex dimensions of the dyadic structure in the sense of Lapidus--van Frankenhuijsen \cite{Lapidus}: equally spaced on $\Re\sigma=1-\Re s$ with period $2\pi/\ln 2$. Equivalently, in the $s$-variable the resummed weight $\sum_k 2^{-k(1-s)}=2^{s-1}/(1-2^{s-1})$ has poles at $s=1+2\pi i n/\ln 2$ --- the zeros of the dyadic Dirichlet polynomial $1-2^{1-s}$, which are exactly the extra zeros of the Dirichlet eta function $\eta(s)=(1-2^{1-s})\zeta(s)$ on the line $\Re s=1$ \cite{Sondow}.

By the Mellin--asymptotic correspondence for harmonic sums \cite{FGD}, transferring \eqref{MellinF} across the tower gives, as $\omega\to\infty$,
\begin{equation}\label{logperiodic}
    F_\phi(\omega)\sim\frac{\omega^{s-1}}{\ln 2}\sum_{n\in\ZZ}\widetilde\phi(\sigma_n)\,e^{2\pi i n\log_2\omega},
\end{equation}
so that $F_\phi(\omega)/\omega^{s-1}$ is asymptotically periodic in $\log_2\omega$ --- a log-periodic oscillation whose Fourier coefficients are the values of the profile's Mellin transform at the complex dimensions.

The arithmetic specialization is now immediate. For a Dirichlet character $\chi$ take the profile
\begin{equation}
    \phi_\chi(\theta)=\sum_{m=1}^\infty\chi(m)e^{-m\theta},\qquad \widetilde\phi_\chi(\sigma)=\Gamma(\sigma)L(\sigma,\chi),
\end{equation}
so that the oscillation amplitudes in \eqref{logperiodic} are Dirichlet $L$-values at the tower:
\begin{equation}\label{Lamplitudes}
    c_n=\frac{\Gamma(\sigma_n)\,L(\sigma_n,\chi)}{\ln 2},\qquad \sigma_n=(1-s)-\frac{2\pi i n}{\ln 2}.
\end{equation}
For the character $\chi_{-4}$ modulo $4$ --- conductor a power of two, so the character period matches the dyadic scale --- the profile has the closed form $\phi_{\chi_{-4}}(\theta)=\tfrac{1}{2}\operatorname{sech}\theta$, with $L(\sigma,\chi_{-4})=\beta(\sigma)$ the Dirichlet beta function, giving $c_n=\Gamma(\sigma_n)\beta(\sigma_n)/\ln 2$. A direct numerical extraction of the oscillation confirms \eqref{Lamplitudes}: the mean and first harmonic agree with $\Gamma(\sigma_n)\beta(\sigma_n)/\ln2$, the higher harmonics being correct but below the numerical floor.

\begin{Remark}[suppression]
The amplitudes decay extremely fast. Since $|\Im\sigma_n|=2\pi|n|/\ln2\approx 9.06\,|n|$ and $|\Gamma(\sigma+i\tau)|\sim e^{-\pi|\tau|/2}$ as $|\tau|\to\infty$, the first harmonic is already of relative size $\sim10^{-6}$ and the remainder is essentially invisible. The log-periodic structure is genuine but, for a single dyadic scale, exponentially small.
\end{Remark}

\begin{Remark}[what is classical here]
Identity \eqref{logperiodic} is an instance of the harmonic-sum Mellin asymptotics of Flajolet--Gourdon--Dumas \cite{FGD}; the appearance of $L$-values in \eqref{Lamplitudes} is the statement that the Mellin transform of an arithmetic profile is an $L$-function. The complex dimensions coincide with the eta extra-zeros. The $L$-values occur at generic points of the line $\Re\sigma=1-\Re s$, not at distinguished arguments. So this route too recovers classical objects --- the analytic theory of harmonic sums and the eta function --- rather than producing new ones.
\end{Remark}

\section{Roads back to classical objects}\label{sec:discussion}

It is worth collecting what the preceding sections show, because one pattern runs through all of them. The dyadic Cauchy-kernel identity is a single compact object, and each specialization returns to a classical landmark:
\begin{enumerate}[label=(\roman*)]
    \item On the unit circle, symmetrized, it is the iterated polylogarithm duplication formula telescoped against Hurwitz's Fourier-series formula for $\zeta(s,x)$ (Lemma~\ref{lem:telescope}, \S\ref{sec:sine}).
    \item At $\omega=\pi$ it gives a representation of $\pi^{s-1}$ through the eta--zeta relation, and with the functional equation, the ratio $\zeta(s)/\zeta(1-s)$ (\S\ref{sec:zeta}).
    \item At $\omega=2\pi p/q$ it gives Hurwitz's formula at rational arguments, in the neighborhood of the multiplication theorem (\S\ref{sec:zeta}).
    \item Through its Mellin transform it reproduces the Hasse--Sondow globally convergent series (\S\ref{sec:mellin}).
    \item Through its discrete scale invariance its complex dimensions are the eta extra-zeros, and the log-periodic amplitudes of the associated harmonic sums are Dirichlet $L$-values (\S\ref{sec:Lvalues}).
\end{enumerate}
The recurrence is not accidental. By Lemma~\ref{lem:telescope} the identity is assembled from the duplication formula and, through Hurwitz's formula and the Mellin apparatus, from the classical analytic theory of $\zeta$ and $L$; any specialization inherits that content. What this note offers, then, is not new theorems but a single vantage point --- a compact dyadic identity from which several familiar results may be read off, together with a map of how they connect. Where a particular display above may not previously have appeared in exactly this form, it is in each case a recombination of the ingredients named, and we make no claim of novelty for it. The value, if any, is organizational: the duplication formula, Hurwitz's formula, the Hasse--Sondow series, the eta extra-zeros, and the harmonic-sum asymptotics are not usually seen as facets of one identity, and here they are.

\end{document}